\documentclass[11pt, reqno]{amsart}

\usepackage{amsthm,amssymb,amstext,amscd,amsfonts,amsbsy,amsrefs,amsxtra,latexsym,amsmath,xcolor,mathrsfs,fancybox,upgreek, soul,url}
\usepackage[english]{babel}
\usepackage[all,cmtip]{xy}
\usepackage[latin1]{inputenc}
\usepackage{cancel}
\usepackage[draft]{hyperref}
\usepackage{comment}
\usepackage{mdframed}
\allowdisplaybreaks
\usepackage{mathtools}
\usepackage{enumerate}
\usepackage{thmtools}
\usepackage{thm-restate}

\DeclarePairedDelimiter\abs{\lvert}{\rvert}%
\DeclarePairedDelimiter\norm{\lVert}{\rVert}%

\makeatletter
\let\oldabs\abs
\def\abs{\@ifstar{\oldabs}{\oldabs*}}
\let\oldnorm\norm
\def\norm{\@ifstar{\oldnorm}{\oldnorm*}}
\makeatother

\newtheorem{theorem}{Theorem}
\newtheorem{lemma}[theorem]{Lemma}

\newtheorem{proposition}[theorem]{Proposition}
\newtheorem{conjecture}[theorem]{Conjecture}

\theoremstyle{definition}

\theoremstyle{remark}

\newtheorem*{remark}{Remark}

\numberwithin{theorem}{section}
\numberwithin{proposition}{section}
\numberwithin{lemma}{section}
\numberwithin{corollary}{section}
\numberwithin{equation}{section}
\numberwithin{conjecture}{section}

\newcommand{\Z}{\mathbb{Z}}
\newcommand{\R}{\mathbb{R}}
\newcommand{\C}{\mathbb{C}}

\def\H{\mathbb{H}}
\renewcommand{\pmod}[1]{\  \,  \left( \mathrm{mod} \,  #1 \right)}

\begin{document}

\title[Asymptotic equidistribution and convexity for Dyson partition ranks ]{Asymptotic equidistribution and convexity for Dyson partition ranks}

\author{Joshua Males}

\address{Mathematical Institute, University of Cologne, Weyertal 86-90, 50931 Cologne, Germany}
\email{jmales@math.uni-koeln.de}

\title{Asymptotic Equidistribution and Convexity for Partition Ranks
}


\maketitle

\begin{abstract}
	We study the Dyson rank function $N(r,t;n)$, the number of partitions with rank congruent to $r$ modulo $t$. We first show that it is monotonic in $n$, and then show that it equidistributed as $n \rightarrow \infty$. Using this result we prove a conjecture of Hou and Jagadeeson on the convexity of $N(r,t;n)$.

\end{abstract}

\section*{acknowledgements}
	The author would like to thank Kathrin Bringmann for helpful comments on previous versions of the paper, as well as Chris Jennings-Shaffer for useful conversations. The author would also like to thank the referee for many helpful comments.

\section{Introduction and statement of results}

A familiar statistic in combinatorics is the number of partitions of an integer $n$, denoted by $p(n)$. The function $p(n)$ has been studied extensively, giving rise to results such as the famous Ramanujan congruences \cite{RamanujanCongruence}. Of particular interest to the current paper is the asymptotic behaviour of the number of partitions, proven by Hardy and Ramanujan in \cite{hardy1918asymptotic}. They showed that as $n \rightarrow \infty$
\begin{equation*}
p(n) \sim \frac{1}{4 n \sqrt{3}} e^{2 \pi \sqrt{\frac{n}{6}}}.
\end{equation*}

Other statistics involving partitions have been introduced since, the most pertinent of which for us is the rank of a partition, defined to be the largest part minus the number of parts. We denote the number of partitions of $n$ with rank $m$ by $N(m,n)$. By standard combinatorial arguments it can be shown that the generating function of $N(m,n)$ is given by (see equation 7.2 of \cite{garvan1988new} for example)
\begin{equation*}
R\left( \zeta; q  \right) \coloneqq \sum_{\substack{n \geq 0 \\ m \in \Z}} N(m,n) \zeta^m q^n = \sum_{n \geq 0} \frac{q^{n^2}}{\left( \zeta q , \zeta^{-1} q; q \right)_n},
\end{equation*}
where $\zeta \coloneqq e^{2 \pi i z}$, $q \coloneqq e^{2 \pi i \tau}$ with $\tau \in \H$ the upper half plane, and $(a;q)_n \coloneqq \prod_{j = 0}^{n-1} (1-aq^j)$. Further, to ease notation we set $(a_1 , a_2; q)_n \coloneqq (a_1;q)_n (a_2;q)_n$. First introduced by Dyson in \cite{DysonRank} as an attempt to describe the Ramanujan congruences combinatorially, the rank statistic has a storied history. For example, we have that
\begin{equation*}
R(-1;q) = 1 + \sum_{n \geq 1} \frac{q^{n^2}}{(1+q)^2 (1+q^2)^2 \cdots (1+q^n)^2}, 
\end{equation*}
which is the famous mock theta function $f(q)$, defined by Ramanujan and Watson in the early twentieth century. 

As a further refinement of $N(m,n)$ we let $N(r,t;n)$ be the number of partitions of $n$ with rank congruent to $r$ modulo $t$. It is well-known that for nonnegative integers $r,t$ we have the following equation that relates the generating function for $N(r,t;n)$ to the generating functions of $p(n)$ and $N(m,n)$ (see e.g. Section 14.3.3 of \cite{bringmann2017harmonic}) 
\begin{equation}\label{Proposition: rewriting N(r,t;n) as p(n) + N(m,n)}
\sum_{n \geq 0} N(r,t;n) q^n =  \frac{1}{t} \left[ \sum_{n \geq 0} p(n) q^n + \sum_{j = 1}^{t-1} \zeta_t^{-rj} R(\zeta_t^j ; q)   \right],
\end{equation}
where $\zeta_t \coloneqq e^{2 \pi i/t}$. 

In \cite{bringmann2010dyson} it was remarked that the results therein may be employed to obtain asymptotics of $N(r,t;n)$. This question was explored by Bringmann in \cite{bringmann2009asymptotics} for odd $t$, via use of the circle method. However, while the formulae obtained therein are stronger than our asymptotics, the present paper requires less strict results and hence we have somewhat shorter proofs. While the theorem we present can be concluded from the results of Bringmann in \cite{bringmann2009asymptotics} for odd $t$, we give results for all $t \geq 2$. We prove the following result.

\begin{restatable}{theorem}{theoremMain}\label{Theorem main}
	For fixed $0 \leq r < t$ and $t \geq 2$ we have that
	\begin{equation*}
	N(r,t;n) \sim \frac{1}{t} p(n) \sim \frac{1}{4 tn \sqrt{3}} e^{2 \pi \sqrt{\frac{n}{6}}}
	\end{equation*}
	as $n \rightarrow \infty$. Hence for fixed $t$ the number of partitions of rank congruent to $r$ modulo $t$ is equidistributed in the limit.
\end{restatable}

Recently, in \cite{bessenrodt2016maximal} Ono and Bessenrodt showed that the partition function satisfies the following convexity result. If $a,b \geq 1$ and $a+b \geq 9$ then 
\begin{equation*}
p(a) p(b) > p(a+b).
\end{equation*}
A natural question to ask is then: does $N(r,t;n)$ satisfy a similar property? In \cite{hou2018dyson} Hou and Jagadeesan provide an answer if $t =3$. They showed that for $0 \leq r \leq 2$ we have
\begin{equation*}
N(r,3;a)N(r,3;b) > N(r,3;a+b)
\end{equation*}
for all $a,b$ larger than some specific bound. Further, at the end of the same paper, the authors offer the following conjecture on a more general convexity result.

\begin{conjecture}\label{Conjecture to be proven}
	For $0 \leq r < t$ and $t \geq 2$ then
	\begin{equation*}
	N(r,t;a) N(r,t;b) > N(r,t;a+b)
	\end{equation*}
	for sufficiently large $a$ and $b$.
\end{conjecture}

As a simple consequence of Theorem \ref{Theorem main} we prove the following theorem.

\begin{restatable}{theorem}{theoremConjecture}\label{Theorem: conjecture is true}
	
	Conjecture \ref{Conjecture to be proven} is true.
\end{restatable}

\begin{remark}
	We note that unlike in \cite{hou2018dyson} our proof of Theorem \ref{Theorem: conjecture is true} does not give an explicit lower bound on $a$ and $b$. To yield such a bound one could employ similar techniques to those in \cite{hou2018dyson}, relying on the asymptotics found in \cite{bringmann2009asymptotics}. However, since \cite{bringmann2009asymptotics} gives results only for odd $t$ one could only find such bounds directly for odd $t$. Further, to find an explicit bound for general $t$ is a difficult problem.
\end{remark}

The paper is organised as follows. In Section \ref{Section: prelims} we give some preliminary results needed for the rest of the paper. We begin by showing the strict monotonicity in $n$ of $N(m,n)$ in Section \ref{Section: strict monotonicity of N(m,n)} which then allows us to prove a monotonicity result of $N(r,t;n)$ in Section \ref{Section: monotonicty of N(r,t;n)}. Section \ref{Section: Asymptotic behaviour of Appell} serves to find the asymptotic behaviour of the level three Appell function. In Section \ref{Section: proof of main theorem} we prove Theorem \ref{Theorem main}. We are then able to conclude Theorem \ref{Theorem: conjecture is true} in Section \ref{Section: proof of conjecture}.

\section{Preliminaries}\label{Section: prelims}

\subsection{Appell functions}\label{Section: prelims Appell functions}
We make extensive use of properties of Appell functions in Section \ref{Section: Asymptotic behaviour of Appell}, and so here we recall relevant results without proof. In his celebrated thesis \cite{zwegers2008mock} Zwegers studied the Appell function 
\begin{equation*}
\mu(u,z ; \tau) \coloneqq \frac{ e^{\pi i u}}{\vartheta(z; \tau)} \sum_{n \in \Z} \frac{ (-1)^n e^{\pi i (n^2 + n) \tau} e^{2 \pi i n z   }}{1 - e^{2 \pi i n \tau} e^{2 \pi i u}},
\end{equation*}
where
\begin{equation*}
\vartheta(z; \tau) \coloneqq \sum_{n \in \frac{1}{2} + \Z} e^{\pi i n^2 \tau + 2 \pi i n (z + \frac{1}{2})},
\end{equation*}
with $z \in \C$, is a Jacobi theta function. It is well-known that $\vartheta$ satisfies the following two transformation formulae (see e.g. \cite{mumford2007tata});

\begin{equation*}
\vartheta(z + 1; \tau) = - \vartheta\left( z; \tau \right),
\end{equation*}
and
\begin{equation*}
\vartheta(z; \tau) = \frac{i}{\sqrt{-i \tau}} e^{ \frac{- \pi i z^2}{\tau}} \vartheta\left( \frac{z}{\tau} ; -\frac{1}{\tau} \right).
\end{equation*}
Zwegers used this to then show that $\mu$ satisfies

\begin{equation*}
\mu(u+1,v ;\tau) = - \mu(u,v;\tau),
\end{equation*}
and
\begin{equation*}
\mu(u,v;\tau) = \frac{-1}{\sqrt{-i \tau}} e^{\frac{\pi i (u - v)^2}{\tau}} \mu\left( \frac{u}{\tau} , \frac{v}{\tau} ; -\frac{1}{\tau}  \right) + \frac{1}{2i} h(u-v; \tau),
\end{equation*} 
where $h$ is the Mordell integral
\begin{equation*}
h(z; \tau) \coloneqq \int_{\R} \frac{e^{\pi i \tau x^2 - 2 \pi z x}}{\cosh(\pi x)} dx.
\end{equation*}

Further, Zwegers showed the following two transformation properties of $h$;

\begin{equation}\label{Equation: transformation h(z;tau)}
h\left( z; \tau \right) = \frac{1}{\sqrt{-i \tau}} e^{ \frac{\pi i z^2}{\tau}} h\left( \frac{z}{\tau} ; -\frac{1}{\tau} \right),
\end{equation}
and
\begin{equation}\label{Equation: shifting h}
h(z; \tau) + e^{- 2 \pi i z - \pi i \tau} h(z+\tau; \tau) = 2e^{- \pi i z - \frac{\pi i \tau}{4}}.
\end{equation}

In more recent work \cite{zwegers2010multivariable} Zwegers introduced Appell functions of higher level and showed that they also exhibit similar transformation formulae. We define the level $\ell$ Appell function by
\begin{equation*}
A_\ell (u,v;\tau) \coloneqq e^{\pi i \ell u} \sum_{n \in \Z} \frac{ (-1)^{\ell n}  q^{\frac{\ell n(n+1)}{2}} e^{2 \pi i n v} }{1 - e^{2 \pi i u}q^n}.
\end{equation*}
Then it is shown that
\begin{equation*}
\begin{split}
A_\ell (u,v;\tau) & = \sum_{k = 0}^{\ell -1} e^{2 \pi i u k} A_1 \left( \ell u, v + k \tau + \frac{\ell-1}{2} ; \ell \tau \right) \\
& = \sum_{k = 0}^{\ell - 1} e^{2 \pi i u k} \vartheta \left( v + k \tau + \frac{\ell-1}{2}; \ell \tau \right) \mu \left( \ell u, v + k \tau + \frac{\ell-1}{2}; \ell \tau \right),
\end{split}
\end{equation*} 
and so $A_\ell$ inherits transformation properties from $\vartheta$ and $\mu$.

\subsection{A bound on $h$}
In Section \ref{Section: Asymptotic behaviour of Appell} we investigate asymptotic properties of $h$, and make use of a bound given in \cite{2018arXiv180511319J}. Proposition 5.2 therein reads as follows.
\begin{proposition}\label{Proposition: bounding h(z;tau)}
	Let $\kappa$ be a positive integer, $\alpha, \beta \in \R$ with $|\alpha| < \frac{1}{2}$ and $-\frac{1}{2} \leq \beta < \frac{1}{2}$, and $z \in \C$ with $\operatorname{Re}(z) >0$. Then
	\begin{equation*}
	\abs{ h\left( \frac{i \beta}{\kappa z} + \alpha ; \frac{i}{\kappa z}  \right) }  \leq 
	\begin{dcases}
	\abs{\sec(\pi \beta) } \kappa^{\frac{1}{2}} \operatorname{Re}\left(\frac{1}{z} \right)^{- \frac{1}{2}} e^{ -\frac{\pi \beta^2}{\kappa} \operatorname{Re}\left(\frac{1}{z}\right) + \pi \kappa \alpha^2 \operatorname{Re}\left( \frac{1}{z} \right)^{-1}  } & \text{ if } \beta \neq - \frac{1}{2}, \\
	\left( 1 + \kappa^{\frac{1}{2}} \operatorname{Re}\left( \frac{1}{z} \right)^{-\frac{1}{2}}  \right) e^{- \frac{\pi}{4\kappa} \operatorname{Re}\left( \frac{1}{z} \right)} & \text{ if } \beta = - \frac{1}{2}.
	\end{dcases}
	\end{equation*}
\end{proposition}
In particular, we will use this to show that all but finitely many terms arising from a particular Appell function are exponentially decaying in the asymptotic limit.

\subsection{Ingham's Tauberian Theorem}
To conclude our main result, we use the following theorem of Ingham \cite{ingham1941tauberian} that gives an asymptotic formula for the coefficients of certain power series.

\begin{theorem}\label{Theorem: Ingham's Tauberian}
	Let $f(q) \coloneqq \sum_{n \geq 0} a(n)q^n$ be a power series with weakly increasing non-negative coefficients and radius of convergence equal to one. If there exist constants $A > 0$, $\lambda,\alpha \in \R$ such that
	\begin{equation*}
	f(e^{-\varepsilon}) \sim \lambda \varepsilon^\alpha e^{\frac{A}{\varepsilon}}
	\end{equation*}
	as $\varepsilon \rightarrow 0^+$, then
	\begin{equation*}
	a(n) \sim \frac{\lambda}{2 \sqrt{\pi}} \frac{ A^{\frac{\alpha}{2} + \frac{1}{4}}  }{n ^{\frac{\alpha}{2} + \frac{3}{4}}} e^{2 \sqrt{An}}
	\end{equation*}
	as $n \rightarrow \infty$.
\end{theorem}

\section{Strict monotonicity of $N(m,n)$}\label{Section: strict monotonicity of N(m,n)}
In this section we show strict monotonicity of $N(m,n)$ for $n \geq 2m+25$. This follows work of Chan and Mao in \cite{chan2014inequalities} in which the following theorem regarding weak monotonicity of $N(m,n)$ is shown.

\begin{theorem}\label{Theorem: Chan/Mao weak monotonicity N(m,n)}
	For all non-negative integers $m$ and positive integers $n$ we have that
	\begin{equation*}
	N(m,n) \geq N(m,n-1),
	\end{equation*}
	except when $(m,n) = (\pm1,7), (0,8), (\pm3 , 11)$ and when $n=m+2$.
\end{theorem} 

First, we state without proof some relevant results, beginning with the following trivial lemma which is an example of the famous Postage Stamp Problem.

\begin{lemma}\label{Lemma: coeffiecnt for n greater than 18}
	The coefficient of $q^n$ with $n \geq 18$ in the expression
	\begin{equation*}
	\sum_{j \geq 0} q^{3j} \sum_{k \geq 0} q^{4k}
	\end{equation*}
	is greater than or equal to two.
\end{lemma}

We also have the following result, see Lemma 10 of \cite{chan2014inequalities}.

\begin{lemma}\label{Lemma: coefficients of 1-q^m-1 / (1-q^2)(1-q^3) are non-neg}
	The expression
	\begin{equation*}
	\frac{1-q^{m+1}}{(1-q^2)(1-q^3)}
	\end{equation*}
	has non-negative power series coefficients for any positive integer $m$.
\end{lemma}

Lemma 9 of \cite{chan2014inequalities} reads as follows.

\begin{lemma}\label{Lemma: writing out (a)_n}
	With  $(a)_n \coloneqq (a;q)_n$, we have that
	\begin{equation*}
	\frac{1-q}{(aq)_1 (q/a)_1} = \sum_{n \geq 0} \sum_{m = -n}^{n} (-1)^{m+n} a^m q^n,
	\end{equation*}
	and
	\begin{equation*}
	\begin{split}
	\frac{1-q}{(aq)_2 (q/a)_2} = & -q + \frac{1}{1-q^3} + \frac{q^2}{1-q^4} + \frac{q^8}{(1-q^3)(1-q^4)} \\
	& + \sum_{m \geq 1} (a^m + a^{-m}) q^m \left( \frac{1-q^{m+1}}{(1-q^2)(1-q^3)} + \frac{q^{m+3}}{(1-q^3)(1-q^4)} \right).
	\end{split}
	\end{equation*}
\end{lemma}

We use results of \cite{chan2014inequalities} to show that, for sufficiently large $n$, the coefficients of $a^m q^n$ in the series
\begin{equation*}
\sum_{n \geq 0} \frac{(1-q) q^{n^2}}{(aq)_n (q/a)_n}
\end{equation*}
are strictly positive for $n \neq m+2$. This then implies the following proposition.
\begin{proposition}\label{Proposition: strict monototnicity N(m,n)}
	For positive $m$ and $n \geq 2m + 25$, or $m = 0$ and $n \geq 30$, we have that
	\begin{equation*}
	N(m,n) > N(m,n-1).
	\end{equation*}
\end{proposition}

\begin{proof}
	
	As in \cite{chan2014inequalities} we define
	\begin{equation*}
	\sum_{m \in \Z} a^m f_{m,k}(q) \coloneqq \frac{1-q}{(aq)_k (q/a)_k}.
	\end{equation*}
	Then
	\begin{equation}\label{Equation: generating function in terms of f_m,k}
	\begin{split}
	\sum_{n \geq 0} \frac{(1-q) q^{n^2}}{(aq)_n (q/a)_n} =& 1 - q + \sum_{n \geq 1} q^{n^2} f_{0,n}(q) \\
	& + \sum_{m \geq 1} \left(a^m + a^{-m} \right) \left(q f_{m,1} (q) + q^4 f_{m,2} (q) + \sum_{n \geq 3} q^{n^2} f_{m,n} (q) \right).
	\end{split}
	\end{equation}
	The main idea of \cite{chan2014inequalities} is to show that these combinations of $f_{m,n} (q)$ have non-negative coefficients of $q^n$ and $a^m q^n$ for $n$ large enough, and away from $n = m+2$ (since $N(n-2,n) = 0$ trivially). Here, we simply observe that for some larger bound on $n$ the coefficients are in fact strictly positive, implying our result.
	
	Concentrating firstly on the first sum in the right-hand side of \eqref{Equation: generating function in terms of f_m,k}, the proof of Lemma 13 in \cite{chan2014inequalities} (correcting a minor error there) gives
	\begin{equation*}
	\sum_{n \geq 1} q^{n^2} f_{0,n} (q) = \sum_{ n \geq 0} (-1)^n q^{n+1} + q^4 \left( -q + \frac{1}{1-q^3} + \frac{q^2}{1-q^4} + \frac{q^8}{(1-q^3)(1-q^4)}   \right) - \sum_{ n \geq 3} q^{n^2 + 1} + \sum_{n \geq 0} b_n q^n,
	\end{equation*}
	for some nonnegative sequence $\{b_n \}_{n \geq 0}$. Thus, if we show that
	\begin{equation*}
	\sum_{ n \geq 0} (-1)^n q^{n+1} + q^4 \left( -q + \frac{1}{1-q^3} + \frac{q^2}{1-q^4} + \frac{q^8}{(1-q^3)(1-q^4)}   \right) - \sum_{n \geq 3} q^{n^2 + 1}
	\end{equation*}
	has strictly positive coefficients of $q^n$ for large enough $n$ then we are done for this term. Expanding the above expression gives
	\begin{equation*}
	\sum_{n \geq 0} q^{2n + 1} - \sum_{n \geq 0} q^{2n+2} - q^5 + \frac{q^4}{1 - q^3} + \frac{q^6}{1 - q^4} + \frac{q^{12}}{(1-q^3) (1-q^4)} - \sum_{n \geq 2} q^{4n^2 + 1} - \sum_{n \geq 1} q^{4n^2 +4n+2}.
	\end{equation*}
	As in \cite{chan2014inequalities} we note that both of the expressions
	\begin{equation*}
	\sum_{n \geq 0} q^{2n + 1}  - \sum_{n \geq 2} q^{4n^2 + 1}, \hspace{20pt} \frac{q^6}{1 - q^4} -\sum_{n \geq 1} q^{4n^2 +4n+2}
	\end{equation*}
	have non-negative coefficients. So, it remains to show that
	\begin{equation}\label{Equation: Lemma 13 coeffs}
	\frac{q^{12}}{(1-q^3) (1-q^4)} - \sum_{n \geq 0} q^{2n+2}
	\end{equation} 
	has strictly positive coefficients for every $n$ large enough. Using Lemma \ref{Lemma: coeffiecnt for n greater than 18} it is easy to see that for $n \geq 30$ the coefficients of $q^n$ in \eqref{Equation: Lemma 13 coeffs} are strictly positive. 
	
	We next consider the second sum in the right-hand side of \eqref{Equation: generating function in terms of f_m,k} i.e. the expression
	\begin{equation*}
	\sum_{m \geq 1} \left(a^m + a^{-m} \right) \left(q f_{m,1} (q) + q^4 f_{m,2} (q) + \sum_{n \geq 3} q^{n^2} f_{m,n} (q) \right),
	\end{equation*}
	and we wish to show that, for $n$ sufficiently large and $ n \neq m+2$, the coefficients of $a^m q^n$ are strictly positive.
	
	Consider first the terms
	\begin{equation*}
	\frac{q(1-q)}{(aq)_1 (q/a)_1} + \frac{q^4 (1-q)}{(aq)_2 (q/a)_2}.
	\end{equation*}
	We now show that these have positive coefficients of $q^n$ for large enough $n$. This will imply that
	\begin{equation*}
	q f_{m,1} (q) + q^4 f_{m,2} (q)
	\end{equation*} also has positive coefficients for large enough $n$ and $m \geq 1$. Unlike in \cite{chan2014inequalities} we do not need to split this into three cases. Then, by Lemma \ref{Lemma: writing out (a)_n},  we want to show that 
	\begin{equation*}
	q \sum_{n \geq 0} \sum_{\substack{m = -n \\ m \neq 0}}^{n} (-1)^{m+n} a^m q^n + q^4 \sum_{m \geq 1} (a^m + a^{-m}) q^m \left( \frac{1-q^{m+1}}{(1-q^2)(1-q^3)} + \frac{q^{m+3}}{(1-q^3)(1-q^4)} \right)
	\end{equation*}
	has positive coefficients for $n$ large enough. By Lemma \ref{Lemma: coefficients of 1-q^m-1 / (1-q^2)(1-q^3) are non-neg} it clearly suffices to choose $n$ such that the coefficient of $q^n$ in
	\begin{equation*}
	\frac{q^{2m+7}}{(1-q^3)(1-q^4)}
	\end{equation*}
	is at least two. By Lemma \ref{Lemma: coeffiecnt for n greater than 18} we see that choosing $n \geq 2m + 25$ will suffice. Therefore the coefficients of $ q^n$ with $n \geq 2m + 25$ and $m \geq 1$ in the expression
	\begin{equation*}
	\sum_{m \geq 1} \left(a^m + a^{-m} \right) \left(q f_{m,1} (q) + q^4 f_{m,2} (q) \right)
	\end{equation*}
	are strictly positive.
	
	From \cite{chan2014inequalities} we have that $\sum_{k \geq 3} q^{k^2} f_{m,k} (q)$ has non-negative coefficients for all $n$, and so we conclude overall that 
	\begin{equation*}
	\begin{split}
	& N(0,n) > N(0,n-1) \text{ for } n \geq 30, \\ 
	& N(m,n) > N(m,n-1) \text{ for } m \geq 1, n \geq 2m + 25.
	\end{split}
	\end{equation*} 
	
\end{proof}

\section{Monotonicity of $N(r,t;n)$}\label{Section: monotonicty of N(r,t;n)}
Using results of the previous section we now prove the following theorem.

\begin{theorem}\label{Theorem: N(r,t;n) weakly increasing}
	Let $0 \leq r < t$ and $n \geq M$ where $M \coloneqq \max(2r+25, 2(t-r)+25)$. Then we have that
	\begin{equation*}
	N(r,t;n) \geq N(r,t;n-1).
	\end{equation*}
\end{theorem}

\begin{proof}

	We first rewrite $N(r,t;n)$ as
	\begin{equation}\label{Equation: writing out N(r,t;n)}
	N(r,t;n) = \sum_{k \in \Z} N(r + kt, n),
	\end{equation}
	in particular noting that this is a finite sum, since for $|r+kt|> n$ we have $N(r+kt,n) = 0$. We differentiate two separate cases, depending on whether $r = 0$ or $r \neq 0$. If $r + kt + 2 \neq n$ for any $k \in \Z$ then we use Theorem \ref{Theorem: Chan/Mao weak monotonicity N(m,n)} directly to conclude that $N(r,t;n) \geq N(r,t;n-1)$. 
	
	Now assume that there exists a term where $r+kt + 2 = n$. First, let $r \neq 0$. We want to show that
	\begin{equation*}
	\sum_{k \in \Z} N(r + kt, n) \geq \sum_{k \in \Z} N(r + kt, n-1).
	\end{equation*}
	Since $N(-m,n) = N(m,n)$ we see that there are at most two terms that vanish on the left-hand side, given by $N(n-2,n)$ and $N(2-n,n)$. Then their counterparts on the right-hand side satisfy $N(n-2,n-1) = N(2-n,n-1) = 1$. Since $r \neq 0$ and $n \geq M$ there must be at least two non-zero intermediate terms e.g. $N(r,n)$ and $N(r-t,n)$. For each of these intermediate terms we apply Proposition \ref{Proposition: strict monototnicity N(m,n)} and conclude our result for $n \geq M$.

	We now turn to the case of $r = 0$. Then \eqref{Equation: writing out N(r,t;n)} becomes
	\begin{equation*}
	N(0,n) + 2 N(t, n) + \dots + 2 N(n-2, n),
	\end{equation*}
	where again the last term vanishes. We want to show that this expression is greater than or equal to
	\begin{equation*}
	N(0,n-1) + 2 N(t, n-1) + \dots + 2 N(n-1, n-1),
	\end{equation*}
	where the last term is equal to two. Then it is enough to use that $N(0,n) \geq N(0,n-1) + 2$ for large enough $n$. Further, it is easy to see that we may adapt the proof of Proposition \ref{Proposition: strict monototnicity N(m,n)} to show that  $\sum_{n \geq 1} f_{0,n} (q) q^{n^2}$ has coefficients strictly greater than one for all $n \geq 42$, implying that
	\begin{equation*}
	N(0,n) \geq N(0,n-1) + 2,
	\end{equation*}
	for $n \geq 42$. For values of $n$ between $1$ and $42$ we test on MAPLE the expression $N(0,n) - N(0,n-1)$ and can show for all $n \geq 15$ we have that $N(0,n) \geq N(0,n-1) + 2$.
	
	Therefore for $n \geq 15$ we have that
	\begin{equation*}
	N(0,t;n) \geq N(0,t;n-1).
	\end{equation*}
	Combining the above arguments finishes the proof.
	
\end{proof}

\section{Asymptotic behaviour of the Appell function $A_3(u, -\tau; \tau)$}\label{Section: Asymptotic behaviour of Appell}
In this section we investigate the asymptotic behaviour of the Appell function $A_3(u,-\tau;\tau)$ when we let $\tau = \frac{i \varepsilon}{2 \pi}$ and $\varepsilon \rightarrow 0^+$. We further impose that $0 < u \leq \frac{1}{2}$ throughout. We prove the following theorem.

\begin{theorem}\label{Theorem: asymptotics of A_3}
	Let $0 < u \leq \frac{1}{2}$ and $\tau = \frac{i \varepsilon}{2 \pi}$. Then
	\begin{equation*}
	A_3(u,-\tau;\tau) \rightarrow 0
	\end{equation*}
	as $\varepsilon \rightarrow 0^+$.
\end{theorem}

\begin{proof}
	Using the transformation formulae given in Section \ref{Section: prelims Appell functions} we rewrite the level three Appell function
	\begin{equation*}
	\begin{split}
	A_3(u,v; \tau) = & \frac{1}{3 \tau} \sum_{k = 0}^{2} e^{\frac{\pi i u (3 u - 2 v)}{\tau}} \vartheta\left( \frac{v}{3 \tau} + \frac{k}{3} ; - \frac{1}{3 \tau} \right) \mu \left( \frac{u}{\tau}, \frac{v}{3 \tau} + \frac{k}{3} ; -\frac{1}{3 \tau} \right) \\
	& + \frac{1}{\sqrt{-12 i \tau}} \sum_{k = 0}^{2} e^{ \pi i \left( \frac{-k^2 \tau}{3} + \frac{6uk - 2vk}{3} - \frac{v^2}{3 \tau} \right)}   \vartheta\left( \frac{v}{3 \tau} + \frac{k}{3} ; - \frac{1}{3 \tau} \right)  h(3u-v-k\tau ; 3\tau).
	\end{split}
	\end{equation*} 
	Specialising to $v = - \tau$ we obtain
	\begin{equation*}\label{Equation: A_3(u,-tau; tau)}
	\begin{split}
	A_3(u,-\tau; \tau) = & \frac{1}{3 \tau} e^{\frac{\pi i u (3 u + 2 \tau)}{\tau}} \sum_{k = 0}^{2}  \vartheta\left( -\frac{1}{3} + \frac{k}{3} ; - \frac{1}{3 \tau} \right) \mu \left( \frac{u}{\tau}, -\frac{1}{3} + \frac{k}{3} ; -\frac{1}{3 \tau} \right) \\
	& + \frac{e^{\frac{- \pi i \tau}{3}}}{\sqrt{-12 i \tau}} \sum_{k = 0}^{2} e^{ \pi i \left( \frac{-k^2 \tau}{3} + \frac{6uk + 2 \tau k}{3} \right)} \vartheta\left( \frac{ k -1}{3}  ; - \frac{1}{3 \tau} \right)  h(3u+ \tau -k\tau ; 3\tau).
	\end{split}
	\end{equation*}
	We write $A_3(u,-\tau; \tau) = S_1 + S_2$ with
	\begin{equation}\label{Equation: main term for asymptotics}
	S_1 \coloneqq \frac{1}{3 \tau} e^{\frac{\pi i u (3 u + 2 \tau)}{\tau}} \sum_{k = 0}^{2}  \vartheta\left( -\frac{1}{3} + \frac{k}{3} ; - \frac{1}{3 \tau} \right) \mu \left( \frac{u}{\tau}, -\frac{1}{3} + \frac{k}{3} ; -\frac{1}{3 \tau} \right)
	\end{equation}
	and
	\begin{equation*}
	S_2 \coloneqq \frac{e^{\frac{- \pi i \tau}{3}}}{\sqrt{-12 i \tau}} \sum_{k = 0}^{2} e^{ \pi i \left( \frac{-k^2 \tau}{3} + \frac{6uk + 2 \tau k}{3} \right)} \vartheta\left( \frac{ k -1}{3}  ; - \frac{1}{3 \tau} \right)  h(3u+ \tau -k\tau ; 3\tau).
	\end{equation*}
	
	We first investigate the terms from $S_1$. By definition we know that
	\begin{equation*}
	\vartheta(z_2; \tau) \mu(z_1,z_2 ; \tau) = e^{\pi i z_1} \sum_{n \in \Z} \frac{ (-1)^n e^{\pi i (n^2 + n) \tau} e^{2 \pi i n z_2   }}{1 - e^{2 \pi i n \tau} e^{2 \pi i z_1}},
	\end{equation*}
	and so
	\begin{equation*}
	\vartheta\left(z; - \frac{1}{3 \tau}\right) \mu\left(\frac{u}{\tau},z ;- \frac{1}{3 \tau}\right) = e^{ \frac{\pi i u}{\tau}} \sum_{n \in \Z} \frac{ (-1)^n e^{- \frac{\pi i (n^2 + n)}{3 \tau} } e^{2 \pi i n z   }}{1 - e^{- \frac{2 \pi i n}{3 \tau}} e^{\frac{2 \pi i u}{\tau}}} = q_0^{- \frac{u}{2}} \sum_{n \in \Z} \frac{(-1)^n q_0^{ \frac{n^2+n}{6}} \zeta^n }{ 1- q_0^{\frac{n}{3} - u}},
	\end{equation*}
	where $q_0 \coloneqq e^{\frac{- 2 \pi i}{\tau}}$.  Thus, with $z = \frac{k - 1}{3}$, we have
	\begin{equation}\label{Equation: main term asymptotics rearranged}
	S_1 = \frac{1}{3 \tau} q_0^{-\frac{1}{2} (3u^2 + u)} e^{2 \pi i u} \sum_{k = 0}^{2}  \sum_{n \in \Z} \frac{(-1)^n q_0^{ \frac{n^2+n}{6}} \zeta^n }{ 1- q_0^{\frac{n}{3} - u}}.
	\end{equation}
	
	First we check the behaviour of $S_1$ at possible poles. Assume $u =  \frac{1}{3}$, so that the $n = 1$ term has a pole of order one. Then \eqref{Equation: main term for asymptotics} is
	\begin{equation*}
	\frac{1}{3 \tau} \rho q_0^{- \frac{1}{3}} \sum_{k = 0}^{2} \sum_{n \in \Z} \frac{(-1)^n  q_0^{\frac{n^2 + n}{6}} e^{\frac{2 \pi i n (k-1)}{3}}}{1 - q_0^{\frac{n-1}{3}}},
	\end{equation*}
	where $\rho \coloneqq e^{\frac{2 \pi i}{3}}$. The only issues are the $n=1$ terms in this sum, and so we investigate the numerator
	\begin{equation*}
	- \frac{1}{3 \tau} \rho \sum_{k = 0}^{2} e^{\frac{2 \pi i (k-1)}{3}}.
	\end{equation*}
	From here it is clear that we have a zero of order one in the numerator and hence have a removable singularity at $u = \frac{1}{3}$. It is clear that for the $n=1$ terms, the limit as $u$ approaches $\frac{1}{3}$ from both above and below is zero, since the numerator is always zero and the denominator is non-zero away from $u = \frac{1}{3}$.	
	
	Furthermore, it is clear that
	\begin{equation*}
	\sum_{k = 0}^{2} \zeta^n = \begin{dcases}
	3 \text{ if } n \equiv 0 \pmod{3}, \\
	0 \text{ else}.
	\end{dcases}
	\end{equation*}
	Thus \eqref{Equation: main term asymptotics rearranged} is equal to
	\begin{equation*}
	\frac{1}{\tau} q_0^{-\frac{1}{2} (3u^2 + u)} e^{2 \pi i u}  \sum_{n \in \Z} \frac{(-1)^n q_0^{ \frac{3 n^2+n}{2}} }{ 1- q_0^{n - u}}.
	\end{equation*}
	We want to find the lowest power of $q_0$ in this sum, since negative powers of $q_0$ give growing terms in the asymptotic limit. Considering only the inner sum without the prefactor, the $n = 0$ term is 
	\begin{equation*}
	\frac{1}{1 - q_0^{-u}} = \frac{- q_0^{u}}{1 - q_0^{u}} = -q_0^{u} - q_0^{2u} + \dots,
	\end{equation*}
	where we have used that $u \in (0,\frac{1}{2}]$ and $\tau \in \H$. It is clear that any term $n \geq 1$ will have terms of order $q_0^{\frac{5}{2}}$ or higher.
	
	When $n < 0$ we have that $n - u <0$ and hence the term
	\begin{equation*}
	\begin{split}
	\frac{(-1)^n   q_0^{\frac{3n^2 +n}{2}}   }{1- q_0^{n - u}} = \frac{ (-1)^{n+1}   q_0^{\frac{3n^2 +n}{2}}  q_0^{u - n}  }{1 -  q_0^{u - n}} = & (-1)^{n+1}   q_0^{\frac{3n^2 +n}{2}} q_0^{u - n} \sum_{j \geq 0} q_0^{j \left(u - n \right)} \\
	= & (-1)^{n+1} q_0^{\frac{3n^2 - n + 2u}{2}} \sum_{j \geq 0} q_0^{j \left(u - n \right)},
	\end{split}
	\end{equation*}
	with the lowest order term $(-1)^{n+1} q_0^{\frac{3n^2 - n + 2u}{2}} $. We note that $\frac{3n^2 - n + 2u}{2} \geq 2 + u$.

	We then see that, for $u \in (0, \frac{1}{2}]$, the most negative power of $q_0$ is given by the $n=0$ term and is
	\begin{equation}\label{Equation: contribution of Appell function first term}
	\begin{split}
	-\frac{1}{\tau} q_0^{-\frac{1}{2} (3u^2 + u)} e^{2 \pi i u}  q_0^u & = -\frac{1}{\tau} q_0^{-\frac{1}{2} (3u^2 - u)} e^{2 \pi i u}.
	\end{split}
	\end{equation}
	Note in particular that for $0<u\leq \frac{1}{6}$ we have that $3u^2 -u<0$ and so here we have a positive power of $q_0$, hence in this case \eqref{Equation: contribution of Appell function first term} tends to $0$ in our asymptotic limit.
	
	We now investigate the second-smallest power of $q_0$ giving a non-zero contribution to the asymptotic behaviour. This is given by the second term in the $n=0$ expansion, and is
	\begin{equation*}
	-\frac{1}{\tau} q_0^{- \frac{1}{2} (3u^2 +u)} e^{2 \pi i u} q_0^{2u} = - \frac{1}{ \tau} q_0^{- \frac{3}{2} (u^2 - u)} e^{2 \pi i u}.
	\end{equation*}
	Since $u^2 - u = u(u-1) < 0$ the power of $q_0$ is positive and hence this term gives vanishing contribution to the asymptotic behaviour. In a similar way, all further terms give no contribution, since the power of $q_0$ increases as we take $\abs{n}$ larger in \eqref{Equation: main term asymptotics rearranged}.

	Now we look to find the contribution of the error of modularity terms $S_2$ to the asymptotic behaviour of $A_3$. First, we note that the smallest power of $q_0$ appearing in $\vartheta\left( \frac{k - 1}{3}  ; - \frac{1}{3 \tau} \right)$ is given by
	\begin{equation}\label{Equation: smallest power q_0 in theta}
	- i q_0^{\frac{1}{24}} \left( e^{- \frac{\pi i (k-1)}{3} } - e^{\frac{\pi i (k-1)}{3}}  \right).
	\end{equation}
	Using \eqref{Equation: transformation h(z;tau)} we find that
	\begin{equation*}
	\begin{split}
	h(3u+ (1 - k)\tau ; 3\tau) & = \frac{1}{\sqrt{-3i \tau}} e^{\frac{\pi i (3u +(1-k)\tau)^2}{3 \tau}   } h\left( \frac{u}{\tau} + \frac{1 - k}{3} ; -\frac{1}{3 \tau} \right) \\
	& = \frac{1}{\sqrt{-3i \tau}} e^{\frac{\pi i \tau(k-1)^2}{3} + \frac{3 \pi i u^2}{ \tau } + 2 \pi i u(1-k)  } h\left( \frac{u}{\tau} + \frac{1 - k}{3} ; -\frac{1}{3 \tau} \right).
	\end{split}
	\end{equation*}
	Hence we have
	\begin{equation}\label{Equation: error terms with inverted h}
	S_2 = \frac{i}{6 \tau} \sum_{k = 0}^{2} e^{ \pi i \left( 2u + \frac{3u^2}{\tau} \right)} \vartheta\left( \frac{k - 1}{3}  ; - \frac{1}{3 \tau} \right)   h\left( \frac{u}{\tau} + \frac{1 - k}{3} ; -\frac{1}{3 \tau} \right).
	\end{equation}
	
	If $u \leq \frac{1}{6}$ we rewrite
	\begin{equation*}
	h\left( \frac{u}{\tau} + \frac{1 - k}{3} ; -\frac{1}{3 \tau} \right) = h \left( \frac{-3u}{- 3 \tau} + \frac{1-k}{3} ; -\frac{1}{3 \tau}\right).
	\end{equation*}
	Then writing $\tau = \frac{i \varepsilon}{2 \pi}$ we see that Proposition \ref{Proposition: bounding h(z;tau)} with $\kappa = 1$, $z = \frac{3 \varepsilon}{2 \pi}$, $\beta = -3u$, and $\alpha = \frac{1-k}{3}$ gives the bound as $\varepsilon \rightarrow 0^+$ of
	\begin{equation*}
	\abs{h\left( \frac{u}{\tau} + \frac{1 - k}{3} ; -\frac{1}{3 \tau} \right)} \leq \begin{dcases}
	\abs{\sec(- 3 \pi u) }  \left(\frac{2 \pi}{3 \varepsilon} \right)^{- \frac{1}{2}} e^{ - \frac{6 \pi^2 u^2}{\varepsilon} + \frac{(1-k)^2 \varepsilon}{6}    } & \text{ if } - 3u \neq - \frac{1}{2}, \\
	\left( 1 + \left( \frac{2 \pi}{3 \varepsilon} \right)^{-\frac{1}{2}}  \right) e^{- \frac{ \pi^2}{6 \varepsilon}} & \text{ if } -3u = - \frac{1}{2}.
	\end{dcases}
	\end{equation*}
	Combining the above we see that for $u < \frac{1}{6}$ the contribution of $S_2$ to the overall asymptotic behaviour is bounded in modulus by
	\begin{equation*}
	\begin{split}
	\frac{ 2 \pi \abs{\sec(- 3 \pi u) }}{3 \varepsilon} \sum_{k = 0}^{2}    e^{-\frac{\pi^2}{6 \varepsilon}}   \left(\frac{2 \pi}{3 \varepsilon} \right)^{- \frac{1}{2}} e^{ \frac{(1-k)^2 \varepsilon}{6}    } .
	\end{split}
	\end{equation*}
	It is easy to see that as $\varepsilon \rightarrow 0^+$ this contribution vanishes. In a similar way, the contribution from $S_2$ to the overall asymptotics vanishes when $u = \frac{1}{6}$.
	
	We now consider $u > \frac{1}{6}$. In order to apply Proposition \ref{Proposition: bounding h(z;tau)} we need to shift the function $h$. Using \eqref{Equation: shifting h} gives
	\begin{equation*}
	\begin{split}
	h\left( \frac{- 3u}{- 3 \tau} + \frac{1-k}{3}; -\frac{1}{3 \tau}  \right) & = h \left( \frac{- 3u + 1}{- 3 \tau} + \frac{1}{3 \tau} + \frac{1-k}{3}; -\frac{1}{3 \tau}  \right) \\
	& = - e^{ - 2 \pi i \left( \frac{u}{\tau} + \frac{1-k}{3} \right) + \frac{\pi i }{3 \tau} } h\left( \frac{1-3u}{- 3 \tau} + \frac{1-k}{3}; -\frac{1}{3 \tau} \right) + 2e^{- \pi i \left( \frac{u}{\tau} + \frac{1-k}{3}   \right) + \frac{\pi i}{12 \tau}} \\
	& = - e^{ - 2 \pi i \left( \frac{u}{\tau} + \frac{1-k}{3} \right) + \frac{\pi i }{3 \tau} } h\left( \frac{1-3u}{- 3 \tau} + \frac{1-k}{3}; -\frac{1}{3 \tau} \right) + 2 e^{ \frac{\pi i (k-1)}{3} } q_0^{\frac{u}{2} - \frac{1}{24}}.
	\end{split}
	\end{equation*}
	Then we write $S_2 =  S_{2,1} + S_{2,2}$, where
	\begin{equation*}
	S_{2,1} \coloneqq \frac{-i}{6 \tau} \sum_{k = 0}^{2} e^{ \pi i \left( 2u + \frac{3u^2}{\tau} \right)} e^{ - 2 \pi i \left( \frac{u}{\tau} + \frac{1-k}{3} \right) + \frac{\pi i }{3 \tau} } \vartheta\left( \frac{k - 1}{3}  ; - \frac{1}{3 \tau} \right)   h\left( \frac{1-3u}{- 3 \tau} + \frac{1-k}{3}; -\frac{1}{3 \tau} \right)
	\end{equation*}
	and
	\begin{equation*}
	\begin{split}
	S_{2,2} \coloneqq & \frac{i}{3 \tau} \sum_{k = 0}^{2} e^{ \pi i \left( 2u + \frac{3u^2}{\tau} \right)} \vartheta\left( \frac{k - 1}{3}  ; - \frac{1}{3 \tau} \right)  e^{ \frac{\pi i (k-1)}{3} } q_0^{\frac{u}{2} - \frac{1}{24}} \\
	= & \frac{i}{3 \tau} \sum_{k = 0}^{2} e^{2 \pi i u} q_0^{-\frac{3u^2}{2} + \frac{u}{2} - \frac{1}{24}} \vartheta\left( \frac{k - 1}{3}  ; - \frac{1}{3 \tau} \right)  e^{ \frac{\pi i (k-1)}{3} }.
	\end{split}
	\end{equation*}

	We concentrate firstly on $S_{2,1}$. Recalling that $u \leq \frac{1}{2}$ and using Proposition \ref{Proposition: bounding h(z;tau)} with $\kappa = 1$, $z = \frac{3 \varepsilon}{2 \pi}$, $\beta = 1-3u$, and $\alpha = \frac{1-k}{3}$ gives the bound
	\begin{equation*}
	\abs{h\left( \frac{1 - 3u}{- 3 \tau} + \frac{1 - k}{3} ; -\frac{1}{3 \tau} \right)} \leq \abs{\sec(\pi (1- 3u)) }  \left(\frac{2 \pi}{3 \varepsilon} \right)^{- \frac{1}{2}} e^{ - \frac{2 \pi^2 (3u-1)^2}{3 \varepsilon} + \frac{(1-k)^2 \varepsilon}{6}   }.
	\end{equation*}
	Then we see that the contribution of $S_{2,1}$ to the overall asymptotic behaviour is bounded in modulus by
	\begin{equation*}
	\frac{2 \pi \abs{\sec(\pi (1- 3u)) }}{3 \varepsilon} \sum_{k = 0}^{2}     e^{\frac{- \pi^2}{6 \varepsilon}}  \left(\frac{2 \pi}{3 \varepsilon} \right)^{- \frac{1}{2}} e^{ \frac{(1-k)^2 \varepsilon}{6}   } .
	\end{equation*}
	It is easy to see that as $\varepsilon \rightarrow 0^+$ this contribution vanishes.
	We are left to consider the contribution of $S_{2,2}$. Using the behaviour of $\vartheta$ given in \eqref{Equation: smallest power q_0 in theta} the lowest power of $q_0$ arising from this sum is
	\begin{equation*}
	\frac{1}{3 \tau} \sum_{k = 0}^{2} e^{2 \pi i u} q_0^{-\frac{3u^2}{2} + \frac{u}{2} - \frac{1}{24}} q_0^{\frac{1}{24}} \left( 1 - e^{ \frac{2 \pi i (k-1)}{3}} \right)  = \frac{1}{\tau} e^{2 \pi i u} q_0^{-\frac{1}{2} (3u^2 - u)},
	\end{equation*}
	exactly canceling the contribution from the first term of the Appell function given in \eqref{Equation: contribution of Appell function first term}. So, when $u > \frac{1}{6}$ we must investigate the second-largest non-zero term of both the Appell function and $S_{2,2}$, since all terms in $S_{2,1}$ are exponentially suppressed in the limit.
	
	It is easily seen from the definition of $\vartheta$ that the power of $q_0$ in $\vartheta\left( \frac{k-1}{3}; -\frac{1}{3\tau}  \right)$ is greater than or equal to $\frac{1}{24}+\frac{1}{3}$ for other terms. Then the power of $q_0$ in $S_{2,2}$ is seen to be positive, since $-\frac{1}{2} (3u^2 -u) \geq -\frac{1}{8}$ for $\frac{1}{6} < u \leq \frac{1}{2}$. Hence these terms give no contribution in the limiting situation. 
	Further, we have already seen that there are no other non-vanishing contributions from \eqref{Equation: main term asymptotics rearranged}. The claimed result now follows.
\end{proof}

\section{Proof of Theorem \ref{Theorem main}}\label{Section: proof of main theorem}
In this section we prove the following.

\theoremMain*

\begin{proof}
	
	From Theorem \ref{Theorem: N(r,t;n) weakly increasing} we know that the power series
	\begin{equation*}
	\sum_{n \geq M} N(r,t;n) q^n
	\end{equation*}
	has weakly increasing coefficients. We are therefore in the situation where we may apply Theorem \ref{Theorem: Ingham's Tauberian}, and so we investigate the asymptotic behaviour
	\begin{equation*}
	\lim\limits_{\varepsilon \rightarrow 0^+} \sum_{n \geq 1} N(r,t;n) e^{- \varepsilon n}.
	\end{equation*}
	Using \eqref{Proposition: rewriting N(r,t;n) as p(n) + N(m,n)} and the fact that $R(\zeta;q) = R(\zeta^{-1};q)$ we have that
	\begin{equation*}
	\sum_{n \geq 0} N(r,t;n)q^n = \frac{1}{t} \left[ \sum_{n = 0}^{\infty} p(n) q^n + \sum_{j = 1}^{\lfloor{\frac{t-1}{2}} \rfloor} \left( \zeta_t^{rj} + \zeta_t^{-rj} \right) R(\zeta_t^j ; q)  + \delta_{t} (-1)^r R(-1;q) \right],
	\end{equation*}
	where $\delta_t \coloneqq 1$ if $t$ is even, and $0$ otherwise.
	We next note that it is possible to rewrite
	\begin{equation*}
	\begin{split}
	R(\zeta; q) = \frac{(1- \zeta)}{(q)_\infty} \sum_{n= -\infty}^{\infty} \frac{ (-1)^n  q^{\frac{n(3n + 1)}{2}} }{1 - \zeta q^n} & = (1- \zeta) \phi(\tau)^{-1} \sum_{n= -\infty}^{\infty} \frac{ (-1)^n  q^{\frac{3n(n + 1)}{2}} q^{-n} }{1 - \zeta q^n} \\
	& = \left( \zeta^{-\frac{3}{2}} - \zeta^{- \frac{1}{2}}  \right) \frac{1}{\phi(\tau)} A_3 (z, -\tau; \tau),
	\end{split}
	\end{equation*}
	where $\phi(\tau) \coloneqq \prod_{n \geq 1} (1-q^n)$.
	
	Considering generating functions we therefore want to investigate the behaviour of
	\begin{equation}\label{Equation: Generating function of rank mod t}
	\frac{1}{t \phi(\tau)} \left[ 1 + \sum_{j = 1}^{\lfloor{\frac{t-1}{2}} \rfloor}  \left( \zeta_t^{rj} + \zeta_t^{-rj} \right) \left( \zeta_{2t}^{-3j} - \zeta_{2t}^{-j}  \right) A_3 \left(\frac{j}{t}, -\tau ; \tau \right) + 2 i \delta_{t} (-1)^r   A_3 \left(\frac{1}{2}, -\tau ; \tau \right)  \right].
	\end{equation}
	Let $\tau = \frac{i \varepsilon}{2 \pi}$ and consider $\varepsilon \rightarrow 0^+$. We use Theorem \ref{Theorem: asymptotics of A_3} with $u = \frac{j}{t}$ and see that the term in square brackets is asymptotically equal to $1$ in this limit. Hence we have that \eqref{Equation: Generating function of rank mod t} behaves as
	\begin{equation*}
	\frac{1}{t \phi\left( \frac{i \varepsilon}{2 \pi} \right)} \sim \frac{1}{t \sqrt{2 \pi}} \varepsilon^{\frac{1}{2}} e^{\frac{\pi^2}{6 \varepsilon}}.
	\end{equation*}
	Then using Theorem \ref{Theorem: Ingham's Tauberian} we see that as $n \rightarrow \infty$
	\begin{equation*}
	N(r,t;n) \sim \frac{1}{t} p(n) \sim \frac{1}{4 tn \sqrt{3}} e^{2 \pi \sqrt{\frac{n}{6}}}. 
	\end{equation*}
	The claim now follows.
\end{proof}

\section{Proof of Theorem \ref{Theorem: conjecture is true}}\label{Section: proof of conjecture}

As a simple application of Theorem \ref{Theorem main} we prove the following theorem.

\theoremConjecture*

\begin{proof}
	Consider the ratio
	\begin{equation*}
	\frac{N(r,t;a) N(r,t;b)}{N(r,t;a+b)}
	\end{equation*}
	as $a,b \rightarrow \infty$. By Theorem \ref{Theorem main} we have
	\begin{equation*}
	\begin{split}
	\frac{N(r,t;a) N(r,t;b)}{N(r,t;a+b)} \sim & \frac{\frac{1}{4 ta \sqrt{3}} e^{2 \pi \sqrt{\frac{a}{6}}}\frac{1}{4 tb \sqrt{3}} e^{2 \pi \sqrt{\frac{b}{6}}}}{\frac{1}{4 t(a+b) \sqrt{3}} e^{2 \pi \sqrt{\frac{(a+b)}{6}}}} = \frac{(4ta \sqrt{3} + 4tb \sqrt{3})}{48 t^2 ab } \frac{e^{2 \pi   \sqrt{\frac{a}{6}} \sqrt{\frac{b}{6}} }}{ e^{2 \pi \sqrt{\frac{(a+b)}{6}}}} > 1
	\end{split}
	\end{equation*} 
	as $a,b \rightarrow \infty$.
\end{proof}

\bibliographystyle{amsplain}
\bibliography{bib}

\end{document}